\documentclass{amsart}

\usepackage{amsmath, amssymb,amscd,  latexsym, graphicx, mathrsfs, enumerate}

\usepackage{mathtools}

\usepackage{color}
\usepackage{booktabs}

\usepackage{hyperref}

\hypersetup{
  colorlinks   = true, 
  urlcolor     = black, 
  linkcolor    = black, 
  citecolor   = black 
}

\newtheorem{theorem}{Theorem}[section]
\newtheorem{prop}[theorem]{Proposition}
\newtheorem{lemma}[theorem]{Lemma}
\newtheorem{corollary}[theorem]{Corollary}

\newtheorem*{theoremA}{Theorem A}
\newtheorem*{corB}{Corollary B}
\newtheorem*{theoremC}{Theorem C}
\newtheorem*{theoremD}{Theorem D}
\newtheorem*{corE}{Corollary E}
\newtheorem*{corF}{Corollary F}
\theoremstyle{definition}
\newtheorem{remark}[theorem]{Remark}

\newcommand{\lra}{\longrightarrow}

\newcommand{\C}{\mathbb{C}}
\newcommand{\R}{\mathbb{R}}
\newcommand{\Q}{\mathbb{Q}}
\newcommand{\Z}{\mathbb{Z}}
\newcommand{\N}{\mathbb{N}}
\newcommand{\F}{\mathbb{F}}
\newcommand{\bQ}{\overline{{\Q}}}
\newcommand{\bF}{\overline{\F}}
\renewcommand{\O}{\mathcal{O}}
\newcommand{\diag}{\mathrm{diag}}
\newcommand{\rad}{\mathrm{rad}}
\newcommand{\br}{\overline{\rho}}

\newcommand{\Cl}{\mathrm{Cl}}

\newcommand{\GSp}{\mathrm{GSp}}
\newcommand{\Sp}{\mathrm{Sp}}
\newcommand{\GL}{\mathrm{GL}}

\newcommand{\SO}{\mathrm{SO}}
\newcommand{\OO}{\mathrm{O}}
\newcommand{\PSL}{\mathrm{PSL}}
\newcommand{\PGL}{\mathrm{PGL}}
\newcommand{\SL}{\mathrm{SL}}
\newcommand{\Sz}{\mathrm{Sz}}
\newcommand{\Sym}{\mathrm{Sym}}
\newcommand{\Gal}{\mathrm{Gal}}
\newcommand{\im}{\mathrm{Im}}
\newcommand{\Aut}{\mathrm{Aut}}
\newcommand{\coker}{\mathrm{coker}}
\newcommand{\End}{\mathrm{End}}
\newcommand{\Bil}{\mathrm{Bil}}
\newcommand{\Alt}{\mathrm{Alt}}

\renewcommand{\geq}{\geqslant}
\renewcommand{\ge}{\geqslant}
\renewcommand{\leq}{\leqslant}
\renewcommand{\le}{\leqslant}

\title[Non-existence of some Galois representations]{The non-existence of some
Galois representations of moderate dimension in small characteristic}

\author{Alexandru Ghitza \and Takuya Yamauchi}

\dedicatory{Dedicated to Yuichiro Taguchi on his 60th birthday}

\keywords{mod $p$ Galois representations, non-existence}

\address{Alexandru Ghitza \\
  School of Mathematics and Statistics\\
  University of Melbourne \\ }
\email{aghitza@alum.mit.edu}

\address{Takuya Yamauchi \\
Mathematical Inst. Tohoku Univ.\\
 6-3, Aoba, Aramaki, Aoba-Ku, Sendai 980-8578, JAPAN}
\email{takuya.yamauchi.c3@tohoku.ac.jp}

\begin{document}

\begin{abstract}
  Refining arguments of Moon \cite{Mo1,Mo2}, under the assumption of the Generalized Riemann Hypothesis, we prove the non-existence of irreducible mod $2$ Galois representations unramified outside $2$ of dimensions $\leq 4$, and of totally real such representations of dimensions $\leq 8$.
  We also prove the non-existence of irreducible totally real mod $3$ representations unramified outside $3$ of dimensions $\leq 4$.

  We show unconditionally that the image of an irreducible mod $2$ representation $\br:G_\Q\lra \GSp_4(\bF_2)$ that is unramified outside $2$ must be large.
  Under GRH, we then deduce the non-existence of such representations.
\end{abstract}

\maketitle

\tableofcontents

\section{Introduction}\label{intro}

Let $K$ be a number field, $d$ a positive integer, and $p$ a prime.
The non-existence of $d$-dimensional mod $p$ Galois representations of
$G_K:=\Gal(\overline{K}/K)$ with small ramification plays an important role in proving Serre's modularity conjecture
(cf. \cite{KW} for odd Galois representations with $K=\Q$ and $d=2$).

We will call a Galois representation $p$-ramified if it is unramified outside $\{p\}$.
The first results on $p$-ramified Galois representations $\br:G_\Q\lra \GL_2(\bF_p)$ are due to Tate \cite{Ta} for $p=2$ and Serre \cite{Se1} for $p=3$, followed by several authors who attacked this kind of problem for small $p,d$, and $K$ with $[K:\Q]\le 2$ (\cite{Mo2},\cite{MT},\cite{S}, \cite{GM}).

In this paper, we prove non-existence results of two distinct flavors, using two different methods.

The first type is conditional on the Generalized Riemann Hypothesis and builds on the methods developed by Moon in \cite{Mo1,Mo2}:
\begin{theoremA}
  Assume GRH.
  There exist no irreducible
  \begin{enumerate}[(a)]
    \item  $2$-ramified representations $\br:G_\Q\lra\GL_d(\bF_2)$ with $d\in\{2,3,4\}$;

    \item totally real $3$-ramified representations $\br:G_\Q\lra\GL_d(\bF_3)$ with $d\in\{2,3,4\}$;

    \item totally real $2$-ramified representations $\br:G_\Q\lra\GL_d(\bF_2)$ with $d\in\{5,6,7,8\}$.
  \end{enumerate}
\end{theoremA}
Note that the case $p=2$, $d\leq 4$ was proved in \cite{Mo2} by Moon without using GRH, but under the assumption that $\Q_{\br}:=\bQ^{\ker(\br)}$ is totally real.

As an application of Theorem A, we show that the case $p=2$, $d=4$ together with a straightforward argument gives
\begin{corB}
  Assume GRH.
  There exist no self-dual $2$-ramified representations $\br:G_\Q\lra \GL_5(\bF_2)$ of length $2$.
\end{corB}
Note that ``of length $2$'' is used here in the representation-theoretic sense, namely that the Galois module has an irreducible submodule and the quotient by it is irreducible.
(This is not to be confused with Moon's notion of $p$-length, which is about finite groups and plays an essential role in the proof of Theorem A.)
We also remark that for each subfield $F$ of $\bF_2$, any $2$-ramified self-dual Galois representation $\br:G_\Q\lra \GL_5(F)$ contains the trivial representation as a direct factor, so ``of length $2$'' is the best we can hope for.

The second flavor of result we present is unconditional and focuses on symplectic representations, starting with the $\GSp_4(\F_2)$-valued case:
\begin{theoremC}
  There exist no representations $\br:G_\Q\lra \GSp_4(\F_2)$ that are irreducible and $2$-ramified.
\end{theoremC}
The proof uses the isomorphism $\GSp_4(\F_2)\simeq S_6$ and associates to $\br$ a rational polynomial of degree at most $6$, approach that may also be of use in the study of Galois representations coming from $2$-torsion points on abelian surfaces, as in \cite{GY}.

In the more general $\GSp_4(\bF_2)$-valued case we show:
\begin{theoremD}
  Let $\br:G_\Q\lra \GSp_4(\bF_2)$ be an irreducible $2$-ramified representation.
  Then $\im(\br)$ is either isomorphic to the Suzuki group $\Sz(\F_{2^r})$ for some odd $r>1$, or is equal to $\Sp_4(\F_{2^s})$ for some $s>1$.

  Hence there exist no irreducible $2$-ramified representations $\br:G_\Q\lra \GSp_4(\bF_2)$ such that $|\im(\br)|<29120$.
\end{theoremD}
We note that the proof may be applicable to representations of $G_K$ for number fields $K\neq\Q$.
The key ingredients are: the classification of maximal subgroups inside $\GSp_4(\bF_2)$, the non-existence results for $2$-dimensional representations over certain quadratic fields, and the Jones--Roberts database of number fields \cite{JR}.
Eventually, the problem is reduced to the case of Theorem C.

Once Theorem D is established, we can use the techniques from the proof of Theorem A to show:
\begin{corE}
  Assume GRH.
  There exist no irreducible $2$-ramified representations $\br:G_\Q\lra \GSp_4(\bF_2)$.
\end{corE}

As an immediate consequence we have:
\begin{corF}
  Let $F$ be a holomorphic Siegel cusp form of degree $2$ and weight $(k_1,k_2)$ with $k_1\ge k_2\geq 1$ corresponding to $\big(\Sym^{k_1-k_2}\otimes\det^{k_2}\big)(\C^2)$.
  Suppose that $F$ is a Hecke eigenform and let $\Q_F$ be the Hecke field of $F$ (the number field generated by the Hecke eigenvalues).
  Suppose also that the cuspidal automorphic representation $\pi_F$ attached to $F$ is unramified outside $2$.
  Let $v$ be a finite place of $\Q_F$ above $2$, let $\F_v$ be the residue field and $\br_{F,v}:G_\Q\lra\GSp_4(\F_v)$ be the semisimplification of the mod $v$ Galois representation attached to $F$.
  We have
  \begin{enumerate}
    \item if $\F_v=\F_2$, then $\br_{F,v}$ is trivial;
    \item if we assume GRH, then $\br_{F,v}$ is trivial.
  \end{enumerate}
\end{corF}
The existence of $\br_{F,v}$ follows from \cite{Wei} when $k_2 \ge 3$ and $F$ is neither a CAP form nor an endoscopic form.
All Siegel holomorphic CAP and endoscopic forms for any weight have been classified (cf. \cite[Sections 3.4 and 3.5]{KWY} for a summary), and one can easily attach $\br_{F,v}$ from the shape of the spinor $L$-function of $F$.
For the lower-weight case ($k_2 \le 2$), we may use the mod $2$ Hasse invariant to raise the weight and thereby apply the former results.

This paper is organized as follows.
After some preparatory work in Section~\ref{prelim}, we build on Moon's methods from \cite{Mo1,Mo2} in Section~\ref{gld} to prove Theorem A, then obtain Corollary B as a consequence in Section~\ref{gl5}.
In Section~\ref{GSp4tri} we consider $\GSp_4(\F_2)$-valued representations and prove Theorem C, and in Section~\ref{GSp4nontri} we extend this to obtain a proof of Theorem D and Corollary E.
The appendix is devoted to detailing the Magma (\cite{BCP}) and SageMath (\cite{Sage}) computations used in this paper; we hope this will facilitate reproducibility.

A crucial ingredient for proving the non-existence of representations, used in many places in this paper as well as in several of the key references, is provided by the discriminant bounds of Odlyzko, Poitou, and Serre, and the related tables of discriminant bounds computed by Odlyzko, see \cite{Od}.
This is also where the dependency on GRH stems from.

\textbf{Acknowledgments.} We would like to thank Miho Aoki for helpful discussions.

\section{Preliminaries}\label{prelim}

\subsection{Number fields of small degree}

The L-functions and modular forms database \cite{LMFDB} includes databases computed by Jones and Roberts \cite{JR} that contain (in particular) all $p$-ramified number fields of degree $n\leq 7$ for $p<102$.
From this we extract:
\begin{prop}[Jones--Roberts]\label{prop:lmfdb}
  The only degrees $n\leq 7$ for which there exist $p$-ramified extensions are:
  \begin{itemize}
    \item for $p=2$: $n\in\{2,4\}$;
    \item for $p=3$: $n\in\{2,3,6\}$.
  \end{itemize}
\end{prop}

As Moon points out in the last paragraph of \cite{Mo2}, this imposes divisibility conditions on the degrees of $p$-ramified extensions:
\begin{corollary}\label{cor:lmfdb}
  There are no $2$-ramified Galois extensions whose degree has prime-to-$2$ part in the set $\{3,5,7\}$.
  There are no $3$-ramified Galois extensions whose degree has prime-to-$3$ part in the set $\{4,5,7\}$.
\end{corollary}
\begin{proof}
  Let $G$ be the Galois group of a $p$-ramified Galois extension $K/\Q$ of degree $n$, let $P$ be its $p$-Sylow subgroup, and consider the fixed field $K^P$.
  The extension $K^P/\Q$ is $p$-ramified and has degree equal to the prime-to-$p$ part of $n$.
  For $p\in\{2,3\}$, the existence of $K^P/\Q$ is constrained by Proposition \ref{prop:lmfdb}.
\end{proof}

\subsection{Modular representations of some finite groups}

Note that $2$ is not contained in the list of forbidden $3$-free parts of degrees in Corollary \ref{cor:lmfdb}, so we cannot rule it out without further assumptions.
To be able to avoid it in Proposition \ref{prop:length1}, we use the following consequence of Brauer's Theorem on the independence of modular characters: if $G$ is a finite group and $p$ is a prime number, the number of (isomorphism classes of) irreducible $\bF_p$-representations of $G$ is equal to the number of conjugacy classes of elements of order coprime to $p$ in $G$.
(See for instance \cite[Corollary 3 to Theorem 42, Section 18.2]{Se2}.)
A well-known immediate implication is that every irreducible $\bF_p$-representation of a $p$-group is trivial.
A more obscure consequence that we will use is:
\begin{lemma}\label{lem:gen_dihedral}
  Let $p$ be an odd prime, $m\geq 1$ an integer, and consider the generalized dihedral group $G=(\Z/p\Z)^m \rtimes (\Z/2\Z)$.
  There are two irreducible $\bF_p$-representations of $G$, and both of them are one-dimensional.
\end{lemma}
\begin{proof}
  We get two irreducible one-dimensional $\bF_p$-representations by composing the canonical surjection $G\to (\Z/2\Z)$ with the two characters of $\Z/2\Z$ given by $1\mapsto \pm 1$.

  To see that this exhausts the irreducible $\bF_p$-representations of $G$, we have to count the number of conjugacy classes of elements of order coprime to $p$.
  But $G$ has
  \begin{itemize}
    \item $1$ conjugacy class of the identity element;
    \item $1$ conjugacy class consisting of all $p^m$ elements of order $2$;
    \item $(p^m-1)/2$ conjugacy classes of elements of order $p$ (each of these classes has cardinality two).
  \end{itemize}
  Therefore Brauer's result gives us precisely two irreducible $\bF_p$-representations.
\end{proof}
\begin{remark}\label{rem:dihedral}
  The same argument and conclusion apply to the case of irreducible $\bF_p$-representations of the dihedral group $D_n=(\Z/n\Z)\rtimes (\Z/2\Z)$ if $n$ is a power of the odd prime $p$.
\end{remark}

\begin{corollary}\label{cor:gen_dihedral}
  Suppose $d\geq 2$.
  \begin{enumerate}[(a)]
    \item There are no irreducible representations $\br\colon G_\Q\lra\GL_d(\bF_p)$ such that the Galois group of $\Q_{\br}$ is a $p$-group.
    \item For $p$ odd, there are no irreducible representations $\br\colon G_\Q\lra\GL_d(\bF_p)$ that are $p$-ramified and such that $\Gal(\Q_{\br}/\Q)\simeq (\Z/p\Z)^m\rtimes(\Z/2\Z)$.
  \end{enumerate}
\end{corollary}
\begin{proof}
  Part (a) follows from the already mentioned fact that a $p$-group has no non-trivial irreducible $\bF_p$-representations.

  Part (b) follows directly from Lemma \ref{lem:gen_dihedral}.
\end{proof}

\subsection{Discriminant upper bounds from \texorpdfstring{$p$}{p}-Sylow structure of Galois groups}

Fix a prime $p$ and let $P$ be a $p$-group (possibly trivial).
The \emph{$p$-length of $P$} is the minimal length of a subnormal series with elementary abelian quotients: in other words, it is the smallest integer $N$ for which there exists a series
\begin{equation}\label{elem_series_1}
  P=P^{(1)}\supsetneq P^{(2)}\supsetneq\dots\supsetneq P^{(N)}\supsetneq P^{(N+1)}=\{1\}
\end{equation}
such that for all $i=1,\dots,N$:
\begin{equation*}
  P^{(i+1)} \text{ is a normal subgroup of } P^{(i)} \text{ and } P^{(i+1)}/P^{(i)}\cong (\Z/p\Z)^{m_i} \text{ for some $m_i$}.
\end{equation*}
We will need a few very simple facts:
\begin{itemize}
  \item $P$ has $p$-length $0$ if and only if it is trivial.
  \item $P$ has $p$-length $1$ if and only if it is elementary abelian.
  \item If $P$ is an abelian $p$-group of $p$-length $N$, then the exponent of $P$ is $p^N$.
  \item Both non-abelian groups of order $p^3$ have $p$-length $2$.
\end{itemize}

Note also that, as Jones observes in \cite[Section 2.2]{Jo}, a shortest series of the type \eqref{elem_series_1} is given by what he calls the Frattini filtration of $P$:
\begin{equation*}
  P\supsetneq \Phi(P)\supsetneq \Phi^2(P)\dots\supsetneq \Phi^N(P)=\{1\},
\end{equation*}
where for any group $H$, $\Phi(H)$ denotes the Frattini subgroup of $H$ (that is, the intersection of all maximal subgroups of $H$).
Therefore the $p$-length $N$ of $P$ is the smallest integer such that $\Phi^N(P)$ is the trivial group.

Given a finite group $G$, we define the \emph{$p$-length of $G$} to be the $p$-length of a $p$-Sylow subgroup of $G$.
(If $p$ does not divide the order of $G$, we say that $G$ has $p$-length $0$.)

Let $K/\Q$ be a $p$-ramified, degree $n$ Galois extension.
Let $N$ denote the $p$-length of the Galois group $G$ of $K/\Q$.
Moon (\cite[Proposition 2.4]{Mo1}) gives the following upper bound for the discriminant of $K$:
\begin{equation*}
  d_K \mid p^{Cn},\quad\text{where $C$ satisfies}\quad C<N+1+\frac{N}{p-1}.
\end{equation*}
Therefore the root discriminant of $K$ satisfies $|d_K|^{1/n}\leq p^C$.

We tighten Moon's upper bound by going back to a more precise expression for $C$ from \cite[Lemma 2.3]{Mo1}.
We reformulate it as:
\begin{lemma}\label{lem:exact_C}
  Let $F/\Q_p$ be a finite unramified extension.
  Let $E/F$ be a totally ramified extension with Galois group $I$.
  Suppose $I$ has $p$-length $N$ and fix a series of the form \eqref{elem_series_1} for the $p$-Sylow subgroup $P$ of $I$, with corresponding positive integers $m_1,\dots,m_N$.
  Let $e_0=[I\colon P]$, let $r$ be the remainder of the division of $e_0$ by $p-1$, and let $e_N$ be the order of $I$.
  Then the different of $E/F$ divides $p^C$, where
  \begin{equation*}
    C=N+1+\frac{N}{p-1}-\left(\frac{1}{e_N}+\frac{1}{p-1}\sum_{i=1}^N \frac{1}{p^{m_i-1}}+\frac{r}{p-1}\Big(\frac{1}{e_0}-\frac{1}{e_N}\Big)\right).
  \end{equation*}
\end{lemma}
\begin{proof}
  The expression from \cite[Lemma 2.3]{Mo1} is
  \begin{equation*}
    C = N+1+\sum_{i=1}^N \frac{\alpha_i-1}{e_{i-1}} - \left(\sum_{i=1}^N \frac{1}{p^{m_i}} + \sum_{i=1}^{N-1} \frac{\alpha_i-1}{e_i} + \frac{\alpha_N}{e_N}\right),
  \end{equation*}
  where
  \begin{equation*}
    e_i = e_0 p^{m_1}\dots p^{m_i},\quad \alpha_i = \left\lfloor \frac{e_{i-1}}{p-1}\right\rfloor + 1\quad \text{for }i=1,\dots,N.
  \end{equation*}
  We observe that since $p\equiv 1\pmod{p-1}$, the integer $r$ is also the remainder of the division of $e_i$ by $(p-1)$ for all $i=1,\dots,N$, so that
  \begin{equation*}
    \alpha_i-1=\frac{e_{i-1}-r}{p-1}\quad\text{for all }i=1,\dots,N.
  \end{equation*}
  Our claim follows from this and some straightforward algebraic manipulations of Moon's formula.
\end{proof}

We can now state a tighter version of \cite[Proposition 2.4]{Mo1}:
\begin{prop}\label{prop:upper_bound_C}
  Let $K/\Q$ be a finite Galois extension of degree $n$, such that the Galois group $G$ has $p$-length $N$.
  Then the $p$-part of the discriminant of $K/\Q$ divides $p^{Cn}$, where
  \begin{equation*}
    C < N + 1 + \frac{N}{p-1} - \left(\frac{1}{n}+\frac{1}{p-1}\sum_{i=1}^N \frac{1}{p^{m_i-1}}\right).
  \end{equation*}
\end{prop}
\begin{proof}
  Identical to the proof of \cite[Proposition 2.4]{Mo1} but using the expression in Lemma \ref{lem:exact_C} and the facts that $e_N\mid n$ and $r\geq 0$.
\end{proof}
(Note that Jones describes in \cite[Proposition 2.3]{Jo} a very similar reformulation of Moon's result.)

Here is the strategy we are going to employ: \emph{a priori} we have no control over the quantity in the parenthesis in Proposition \ref{prop:upper_bound_C}, so we simply use
\begin{equation*}
  C<N+1+N/(p-1).
\end{equation*}
This gives us a bound on $|d_K|^{1/n}$, hopefully small enough to appear in Odlyzko's tables.
The latter give us an upper bound on $n$, say $n\leq n_{\text{max}}$.
At this point we return to the expression in Proposition \ref{prop:upper_bound_C} and minimize the quantity
\begin{equation*}
  \frac{1}{n}+\frac{1}{p-1}\sum_{i=1}^N \frac{1}{p^{m_i-1}}
\end{equation*}
under the additional constraint $n\leq n_{\text{max}}$ (together with other arithmetic constraints).
This results in a smaller upper bound for $C$, hence a smaller upper bound on $|d_K|^{1/n}$, and we can start the process over.
In the next section we describe how taking these successive minima over increasingly smaller sets allows us to rule out the existence of the number field $K$.

\section{Non-existence of \texorpdfstring{$\GL_d(\bF_p)$}{GLdFpbar}-valued representations (Thm A)}\label{gld}

Let $d\geq 2$.

Given an irreducible $p$-ramified Galois representation $\br:G_\Q\lra \GL_d(\bF_p)$, we let $K$ be the field fixed by the kernel of $\br$ and $G$ be the Galois group of $K/\Q$.
We also let $n=[K:\Q]=|G|$.

We gather some consequences of results of Harbater, Hoelscher, and Pollak regarding divisibility by powers of $p$:
\begin{prop}\label{prop:divisibility}
  With the above assumptions:
  \begin{enumerate}[(a)]
    \item if $p=2$, then $16$ divides $n$;
    \item assuming GRH, if $p=3$, then $18$ or $27$ divides $n$.
  \end{enumerate}
\end{prop}
\begin{proof}
  \
  \begin{enumerate}[(a)]
    \item $K$ is a $2$-ramified Galois extension.
      By \cite[Theorem 2.23]{Ha}, either $G$ is a $2$-group (ruled out by Corollary \ref{cor:gen_dihedral}), or $16$ divides $n$.
    \item Here $K$ is a $3$-ramified Galois extension, and the analogue of Harbater's result is \cite[Corollary E]{Ho}.
      Unfortunately, it only guarantees divisibility of $n$ by $9$.
      We give a straightforward modification of Hoelscher's argument that improves the divisibility conclusion, at the cost of assuming GRH.

      If $G$ is solvable, \cite[Corollary D]{Ho} says that $27$ divides $n$, or $G$ is cyclic (ruled out by the irreducibility of $\br$), or $G/p(G)\simeq \Z/2\Z$.
      So in the solvable case, we get that $27$ divides $n$, or $G/p(G)\simeq \Z/2\Z$, in which case $n$ is even, but also divisible by $9$ by the original form of \cite[Corollary E]{Ho}.

      If $G$ is non-solvable, \cite[Theorem 2.10]{Po} forces $n\geq 660$.
      Assuming GRH, Odlyzko's Table 1 implies that $|d_K|^{1/n}\geq 27.328$.
      On the other hand, \cite[Theorem III.2.6]{Ne} gives the upper bound
      \begin{equation*}
        |d_K|^{1/n}< 3^{1+v_3(n)},
      \end{equation*}
      and the combination implies $v_3(n)>2.01$, hence $27\mid n$.\qedhere
  \end{enumerate}
\end{proof}

%
%
%

The case where the Galois group $G$ has $p$-length $0$ (so that $K$ is tamely ramified at $p$) is easily dealt with:
\begin{prop}\label{prop:length0}
  Let $K/\Q$ be a finite Galois extension whose Galois group $G$ has $p$-length $0$.
  Let $d\geq 2$.

  If $p\in\{2,3\}$, there exist no irreducible representations $\br:G_\Q\lra\GL_d(\bF_p)$ with $\Q_{\br}=K$ and $p$-ramified.
\end{prop}
\begin{proof}
  With $N=0$ and $p=2$, we get from Proposition \ref{prop:upper_bound_C} that $|d_K|^{1/n}<2$, and then from Odlyzko's (unconditional) Table 2 that $n<3$, so the Galois group is abelian and hence has no irreducible representations of dimension $>1$.

  With $p=3$ we get $|d_K|^{1/n}<3$, so Odlyzko's Table 2 gives $n<4$, which we rule out in the same manner.
\end{proof}

\begin{prop}\label{prop:length1}
  Let $K/\Q$ be a finite Galois extension whose Galois group $G$ has $p$-length $1$.
  Let $d\geq 2$.

  If $p\in\{2,3\}$, there exist no irreducible representations $\br:G_\Q\lra\GL_d(\bF_p)$ with $\Q_{\br}=K$ and $p$-ramified.
\end{prop}
\begin{proof}
  For $p=2$ and $N=1$ we have from Proposition \ref{prop:upper_bound_C}:
  \begin{equation*}
    C < 3 - \left(\frac{1}{n}+\frac{1}{2^{m_1-1}}\right).
  \end{equation*}
  This gives an initial bound $|d_K|^{1/n}<2^3=8$, for which Odlyzko's Table 2 produces the constraint $n<14$.
  Since $16\mid n$ by Proposition \ref{prop:divisibility}(a), we are done.

  For $p=3$ and $N=1$ we have
  \begin{equation*}
    C < \frac{5}{2} - \left(\frac{1}{n}+\frac{1}{2\cdot 3^{m_1-1}}\right),
  \end{equation*}
  giving $|d_K|^{1/n}<3^{5/2}<15.589$, and from Odlyzko's Table 2 we get $n<80$.

  But $n$ must be a multiple of $18$ or $27$ and its prime-to-$3$ part cannot be in $\{4,5,7\}$ by Corollary \ref{cor:lmfdb}.
  The prime-to-$3$ part can also not be $1$ (by Corollary \ref{cor:gen_dihedral}(a)) or $2$: since the $3$-length is $1$, the $3$-Sylow subgroup of $G$ is elementary abelian, so if the prime-to-$3$ part is $2$ then $G\cong (\Z/3\Z)^m\rtimes(\Z/2\Z)$, therefore irreducible representations of dimension $\geq 2$ are ruled out by Corollary \ref{cor:gen_dihedral}(b).

  This leaves as only possibility $n=72$, which gives us the tighter bound
  \begin{equation*}
    C<\frac{5}{2}-\frac{13}{72}<2.320.
  \end{equation*}
  We deduce that $|d_K|^{1/n}<12.792$ and thus $n<40$, and we are done.
\end{proof}

\begin{prop}\label{prop:length2}
  Let $K/\Q$ be a finite Galois extension whose Galois group $G$ has $p$-length $2$.
  Let $d\geq 2$.

  Assuming GRH:
  \begin{enumerate}[(a)]
    \item if $p=2$, there exist no irreducible representations $\br:G_\Q\lra\GL_d(\bF_2)$ with $\Q_{\br}=K$ and $2$-ramified;
    \item if $p=3$, there exist no irreducible representations $\br:G_\Q\lra\GL_d(\bF_3)$ with $\Q_{\br}=K$ totally real and $3$-ramified.
  \end{enumerate}
\end{prop}
\begin{proof}
  \begin{enumerate}[(a)]
    \item Proposition \ref{prop:upper_bound_C} gives
      \begin{equation*}
        C < 5 - \left(\frac{1}{n}+\frac{1}{2^{m_1-1}}+\frac{1}{2^{m_2-1}}\right),
      \end{equation*}
      hence an initial bound $|d_K|^{1/n}<2^5=32$.
      Using Odlyzko's (GRH-conditional) Table 1 we get $n<4800$.
      But $n$ must be a multiple of $16$, and its prime-to-$2$ part cannot be in $\{1,3,5,7\}$.
      This leaves us with a set of $271$ possibilities
      \begin{equation*}
        n\in\{144,176,208,\dots,4768,4784\}.
      \end{equation*}
      A computer search (see Appendix) finds the minimum $865/4608$ for $n=4608$, improving our bound on $C$ to $C<4.813$, which in turn gives $|d_K|^{1/n}<28.110$ and $n<840$.
      The further iterations are summarized in Table \ref{tab_length2_p2}.

      \begin{table}[h]
        \begin{tabular}{cccccl}
          \toprule
          $n<\dots$ & $\min$     & $C<\dots$ & $|d_K|^{1/n}<\dots$ \\ \midrule
          $\infty$  & ?          & $5$       & $32$                \\
          $4800$    & $865/4608$ & $4.813$   & $28.110$            \\
          $840$     & $417/832$  & $4.499$   & $22.612$            \\
          $200$     & $177/176$  & $3.995$   & $15.945$            \\
          $56$                                                     \\
          \bottomrule
        \end{tabular}
        \caption{Iterative lowering of degree bound, length $2$, $p=2$}
        \label{tab_length2_p2}
      \end{table}

      At this point we have $n<56$ and the set of suitable $n$ is empty.

    \item
      Initially we have
      \begin{equation*}
        C < 4 - \left(\frac{1}{n}+\frac{1}{2\cdot 3^{m_1-1}}+\frac{1}{2\cdot 3^{m_2-1}}\right),
      \end{equation*}
      giving $|d_K|^{1/n}<3^4=81$.
      From Odlyzko's Table 1, we see that (under GRH, and for $K$ totally real) $n<280$.
      But, as in Proposition \ref{prop:length1}, $n$ must be divisible by $18$ or $27$, and its prime-to-$3$ part cannot be in $\{1,4,5,7\}$, so we get the list of possibilities
      \begin{equation*}
        n\in\{18, 54, 72, 90, 126, 144, 162, 180, 198, 216, 234, 252, 270\}.
      \end{equation*}
      A computer search (see Appendix) gives iteratively the results in Table \ref{tab_length2_p3}.

      \begin{table}[h]
        \begin{tabular}{cccccl}
          \toprule
          $n<\dots$ & $\min$   & $C<\dots$ & $|d_K|^{1/n}<\dots$ \\ \midrule
          $\infty$  & ?        & $4$       & $81$                \\
          $280$     & $55/162$ & $3.661$   & $55.814$            \\
          $88$      & $37/54$  & $3.315$   & $38.165$            \\
          $40$      & $19/18$  & $2.945$   & $25.417$            \\
          $21$                                                   \\
          \bottomrule
        \end{tabular}
        \caption{Iterative lowering of degree bound, length $2$, $p=3$}
        \label{tab_length2_p3}
      \end{table}

      At this point the set of candidates for $n$ is not yet empty!
      We are left with the case $n=18$, which we now rule out.
      There are three (isomorphism classes of) non-abelian groups of order $18$:
      \begin{equation*}
        S_3\times (\Z/3\Z),\, (\Z/3\Z)^2\rtimes (\Z/2\Z),\, (\Z/9\Z)\rtimes (\Z/2\Z).
      \end{equation*}
      For the first two, a $3$-Sylow subgroup is isomorphic to $(\Z/3\Z)^2$, hence the $3$-length is $1$, which rules them out given the assumptions.
      The last group is the dihedral group $D_9$, which we can rule out using Remark \ref{rem:dihedral}.
      \qedhere
  \end{enumerate}
\end{proof}

\begin{prop}\label{prop:length3}
  Let $K/\Q$ be a finite Galois extension whose Galois group $G$ has $2$-length $3$.
  Let $d\geq 2$.

  Assuming GRH, there exist no irreducible totally real $2$-ramified representations $\br:G_\Q\lra\GL_d(\bF_2)$ with $\Q_{\br}=K$.
\end{prop}
\begin{proof}
  We start with
  \begin{equation*}
    C < 7 - \left(\frac{1}{n}+\frac{1}{2^{m_1-1}}+\frac{1}{2^{m_2-1}}+\frac{1}{2^{m_3-1}}\right),
  \end{equation*}
  giving $|d_K|^{1/n}<2^7=128$.
  From Odlyzko's Table 1, assuming GRH and that $K$ is totally real, we get that $n<4800$.
  We also know that $n$ must be divisible by $16$ and have prime-to-$2$ part not in $\{1,3,5,7\}$.

  A computer search (see Appendix) gives the results in Table \ref{tab_length3}.

  \begin{table}[h]
    \begin{tabular}{cccccl}
      \toprule
      $n<\dots$ & $\min$      & $C<\dots$ & $|d_K|^{1/n}<\dots$ \\ \midrule
      $\infty$  & ?           & $7$       & $128$               \\
      $4800$    & $3457/4608$ & $6.250$   & $76.110$            \\
      $220$     & $521/208$   & $4.496$   & $22.565$            \\
      $18$                                                      \\
      \bottomrule
    \end{tabular}
    \caption{Iterative lowering of degree bound, length $3$, $p=2$}
    \label{tab_length3}
  \end{table}

  At this point we have exhausted the possibilities for $n$.
\end{proof}

\begin{theoremA}
  Assume GRH.
  There exist no irreducible
  \begin{enumerate}[(a)]
    \item  $2$-ramified representations $\br:G_\Q\lra\GL_d(\bF_2)$ with $d\in\{2,3,4\}$;

    \item totally real $3$-ramified representations $\br:G_\Q\lra\GL_d(\bF_3)$ with $d\in\{2,3,4\}$;

    \item totally real $2$-ramified representations $\br:G_\Q\lra\GL_d(\bF_2)$ with $d\in\{5,6,7,8\}$.
  \end{enumerate}
\end{theoremA}
\begin{proof}
  Follows from the previous Propositions, since the $p$-length of any finite subgroup of $\GL_d(\bF_p)$ is $\leq\lceil \log_2(d)\rceil$, as shown at the beginning of \cite[Section 3]{Mo1}.
\end{proof}

\section{The self-dual five-dimensional case (Cor B)}\label{gl5}

In this section we prove
\begin{corB}
  Assume GRH.
  There exist no self-dual $2$-ramified representations $\br:G_\Q\lra \GL_5(\bF_2)$ of length $2$.
\end{corB}

Assume $\br$ takes the values in $F=\F_{2^r}$ for an integer $r\ge 1$.
For a non-degenerate
symmetric $F$-bilinear map $B:F^{\oplus 2n+1}\times F^{\oplus 2n+1}\lra F$, we define
\begin{equation*}
  \OO(B)(F):=\{g\in \GL_{2n+1}(F)\ |\ B(gx,gy)=B(x,y),\ \forall x,y\in F^{\oplus 2n+1}\}.
\end{equation*}
By \cite[p.23, Theorem 1.5.41]{BHRD}, $\OO(B)(F)\simeq \OO(J)(F)$ where
$J$ is the anti-diagonal matrix whose anti-diagonal entries are all 1.
By direct computation, it is easy to see that each element of $\OO(J)(F)$ can be written as
$\left(\begin{array}{ccc}
      A & 0 & B \\
      0 & 1 & 0 \\
      C & 0 & D
    \end{array}
  \right)$
with
$\left(\begin{array}{cc}
      A & B \\
      C & D
    \end{array}
  \right)\in \Sp_{2n}(F)$. This is a special feature in the case of even characteristic.
Thus, $\OO(B)(F)\simeq \OO(J)(F)\simeq \Sp_{2n}(F)$.

We are now ready to prove Corollary B.
\begin{proof}
  Since $\br$ is $2$-ramified, the determinant representation
  $\det(\br)$ factors though pro-$2$-group $\Gal(\Q(\zeta_{2^\infty})/\Q)$.
  Therefore, $\det(\br)$ has to be trivial since its image inside $\bF^\times_2$ has odd cardinality.
  Thus, by assumption, $\br^\vee\simeq \br$.
  It yields
  \begin{multline*}
    F\subset \End_{F[G_\Q]}(\br)\simeq \Bil_{F[G_\Q]}(\br\times \br^\vee,F)
    \simeq
    \Bil_{F[G_\Q]}(\br\times \br,F)= \\
    \Sym_{F[G_\Q]}(\br\times \br,F)\oplus
    \Alt_{F[G_\Q]}(\br\times \br,F),
  \end{multline*}
  where $\Bil$, $\Sym$, resp.\ $\Alt$ denote the spaces of all $F[G_\Q]$-bilinear, symmetric bilinear, resp.\ alternating bilinear forms on the corresponding $F[G_\Q]$-modules.
  Since $\br$ has odd dimension, $\Alt_{F[G_\Q]}(\br\times \br,F)=0$.
  Thus, $\Sym_{F[G_\Q]}(\br\times \br,F)$ contains an element
  which comes from $F^\times \subset (\End_{F[G_\Q]}(\br))^\times$
  and
  it means that there exists a non-degenerate symmetric form $B$ such that
  \begin{equation*}
    \im(\br)\subset \OO(B)(F)\simeq \OO(J)(F)\subset \GL_5(F).
  \end{equation*}
  By using the description of $\OO(J)(F)$, there exists
  a mod 2 Galois representation $\br_1:G_\Q\lra \Sp_4(F)\subset \GL_4(F)$ such that
  $\br\simeq \br_1\oplus \mathbf{1}$.
  Since $\br$ is of length 2, $\br_1$ is irreducible over $F$. Thus, one can apply Theorem A to deduce the claim.
\end{proof}

\section{Non-existence of \texorpdfstring{$\GSp_4(\F_2)$}{GSp4F2}-valued representations (Thm C)}\label{GSp4tri}
In this section, we show the non-existence of irreducible $2$-ramified representations $\br:G_\Q\lra \GSp_4(\F_2)$.
The method may be useful for $G_K$ for some other number fields $K$ and has potential
applications to representations arising from the study of $2$-torsion points on abelian surfaces over $K$.

For each integer $n\ge 1$, $S_n$ stands for the $n$-th symmetric group.
The following result is due to Miho Aoki with a minor modification:
\begin{lemma}\label{poly}Let $K/\Q$ be a Galois extension with Galois group $G$
  and an injective group homomorphism $\iota:G\hookrightarrow S_n$.
  There exists a separable polynomial $f(x)\in\Q[x]$ of degree at most $n$
  such that $K=\Q_f$, the splitting field of $f$. Further, if $G$ acts transitively on $\{1,2,\ldots,n\}$, then $f(x)$ is irreducible over $\Q$ of degree $n$.
\end{lemma}
\begin{proof}Put $H:=\iota(G)\subset S_n$.
  For each $i$ with $1\le i\le n$, let $H_i$ be the stabilizer of $i$ and put
  $G_i=\iota^{-1}(H_i)\subset G$. We also define
  \begin{equation*}
    N_i:=\bigcap_{\sigma\in G}\sigma G_i \sigma^{-1}.
  \end{equation*}
  Let us consider the orbit decomposition
  \begin{equation*}
    \{1,2,\ldots,n\}=\coprod_{k=1}^r O_H(i_k).
  \end{equation*}
  For each $i_k$,
  put $K_{i_k}:=K^{G_{i_k}}$ and choose $\alpha_{i_k}\in K$ such that
  $K_{i_k}=K^{G_{i_k}}=\Q(\alpha_{i_k})$. For each $\sigma \in G$, $\sigma(K_{i_k})=
    K^{\sigma G_{i_k}\sigma^{-1}}$, so the composite of all conjugates of $K_{i_k}$
  is $K^{N_{i_k}}$.
  Let $f_{i_k}(x)$ be the minimal polynomial of $\alpha_{i_k}$ and define
  \begin{equation}\label{g(x)}
    \widetilde{f}(x):=\prod_{k=1}^r f_{i_k}(x).
  \end{equation}
  It is easy to see that
  \begin{equation*}
    \deg(\widetilde{f}(x))=\sum_{k=1}^r[K_{i_k}:\Q]=\sum_{k=1}^r[G:G_{i_k}]=\sum_{k=1}^r[H:H_{i_k}]=
    \sum_{k=1}^r|O_H(i_k)|=n.
  \end{equation*}
  We show that $\Q_{\widetilde{f}}=K$: since $\iota$ is injective, we have
  \begin{equation*}
    \bigcap_{k=1}^r N_{i_k}\subset \bigcap_{j=1}^n G_{j}=\iota^{-1}(\bigcap_{j=1}^n H_j)
    =\iota^{-1}(\bigcap_{k=1}^r H_{i_k})=\iota^{-1}(\{1\})=\{1\},
  \end{equation*}
  therefore
  \begin{equation*}
    \Q_{\widetilde{f}}=\prod_{k=1}^r K^{N_{i_k}}=K^{\bigcap_{k=1}^r N_{i_k}}=K.
  \end{equation*}
  It remains to deal with the (possible) non-separability of $\widetilde{f}$.
  If there exist $k,k'\in\{1,\dots,r\}$ such that $\alpha_{i_k}$ is conjugate under the Galois action of $G$ to $\alpha_{i_{k'}}$, then $f_{i_k}(x)=f_{i_{k'}}(x)$, so we can keep $f_{i_k}(x)$ and discard $f_{i_{k'}(x)}$.
  We obtain a subset $S$ of $\{1,\ldots, r\}$ such that
  \begin{equation*}
    f(x):=\prod_{k\in S}f_{i_k}(x)
  \end{equation*}
  is separable and $\Q_f=\Q_{\widetilde{f}}=K$. Clearly
  $\deg(f(x))\le \deg(\widetilde{f}(x))=n$.

  The latter claim follows easily from the construction.
\end{proof}

Fix the isomorphism from \cite[Section 2]{GY}:
\begin{equation}\label{iden}
  S_6\simeq \GSp_4(\F_2).
\end{equation}
\begin{prop}\label{abs-case}
  Given an absolutely irreducible representation
  \begin{equation*}
    \br:G_\Q\lra \GSp_4(\F_2),
  \end{equation*}
  there exists an irreducible polynomial $f(x)$ over $\Q$ of degree 5 or 6 such that $\Q_{\br}=\Q_f$
  and $\br\simeq \tau_f$ as a representation over $\F_2$ of $G_\Q$ where $\tau_f$ is given by permuting the roots of $f$ and using (\ref{iden}):
  \begin{equation*}
    \tau_f:\Gal(\Q_{f}/\Q)\hookrightarrow S_6\simeq \GSp_4(\F_2).
  \end{equation*}
\end{prop}
\begin{proof}
  By Lemma \ref{poly} there exists a separable polynomial $f(x)$ over $\Q$ such that $\Q_{\br}=\Q_f$.
  By construction and (\ref{iden}), we see easily that $\br=\tau_f$.

  Consider $H:=\im(\br)$ as a subgroup of $S_6$ by the identification (\ref{iden}).
  Using Magma (see Appendix \ref{app-calc}) we find that the absolute irreducibility assumption restricts the order of $H$ as follows:
  \begin{equation*}
    |H|\in\{20, 36, 60, 72, 120, 360, 720\}.
  \end{equation*}
  We start by considering the cases where $|H|$ is divisible by $5$.
  Then
  \begin{equation*}
    H\cong F_{20}, A_5, S_5, A_6, \text{ or }S_6,
  \end{equation*}
  where $F_{20}=C_5\rtimes C_4$ is the Frobenius group of order $5(5-1)$.
  Also, either $H$ stabilizes one element of $\{1,2,3,4,5,6\}$, or it acts transitively on $\{1,2,3,4,5,6\}$.
  In the former case, $f(x)$ is irreducible of degree $5$; in the latter, it is irreducible of degree $6$.

  It remains to consider the cases where $|H|$ is not divisible by $5$: $|H|=36, 72$.
  Up to conjugacy, there are three groups with
  $|H|=36$ and a single group with $|H|=72$.
  There are two possibilities:
  \begin{itemize}
    \item $H$ acts transitively on $\{1,2,3,4,5,6\}$; as above, this means that $f(x)$ is irreducible of degree $6$.
      (This is clearly the case for $|H|=72$, and Magma indicates that it also occurs for the second and third conjugacy classes with $|H|=36$.)
    \item $H$ is conjugate to $\Aut(\{1,2,3\})\times \Aut(\{4,5,6\})$ in $S_6$; then $\br$ is reducible as observed in \cite[Section 3.1]{TY}, so we may discard this case (which does occur for the first conjugacy class with $|H|=36$).\qedhere
  \end{itemize}
\end{proof}

\begin{prop}\label{irred_to_abs}
  Any irreducible $2$-ramified representation $\br\colon G_\Q\lra\GSp_4(\F_2)$ is absolutely irreducible.
\end{prop}
\begin{proof}
  First note that $\br$ is semi-simple over $\bF_2$: by \cite[Theorem 5.17]{La} we have
  \begin{equation*}
    \rad(\bF_2[\im(\br)])=\rad(\F_2[\im(\br)])\otimes_{\F_2}\bF_2,
  \end{equation*}
  and $\rad(\F_2[\im(\br)])=0$ since $\br$ is irreducible over $\F_2$.

  If $\br$ is not absolutely irreducible, by \cite[Proposition 3.2]{TY},
  its irreducible components (over $\bF_2$) are of dimension less than or equal to $2$.

  Let $\br_1:G_\Q\lra \bF^\times_2$ be a one-dimensional component.
  Since it is unramified outside $2$, it factors through
  the profinite $2$-group $\Gal(\Q(\zeta_{2^\infty})/\Q)$.
  On the other hand, its image lies in $\F^\times_{2^n}$ for some $n\in\N$, so has odd order.
  Therefore $\br_1$ must be trivial.

  Insofar as two-dimensional irreducible components (over $\bF_2$) of $\br$ are concerned, \cite[Theorem]{Ta} shows that they (that is, irreducible $2$-ramified two-dimensional mod $2$ Galois representations) do not exist.

  Thus $\br$ is trivial over $\bF_2$ and hence also over $\F_2$.
  This contradicts the irreducibility over $\F_2$.
\end{proof}

\begin{theoremC}
  There exist no representations $\br:G_\Q\lra \GSp_4(\F_2)$ that are irreducible and $2$-ramified.
\end{theoremC}
\begin{proof}
  By Proposition \ref{irred_to_abs}, we may assume that $\br$ is absolutely irreducible.
  We can therefore apply Proposition \ref{abs-case} and get an irreducible polynomial $f(x)$ over $\Q$ of degree $5$ or $6$ such that $\tau_f\simeq\br$.
  Then $\Q_f$ is a quintic or sextic $2$-ramified number field, which is ruled out by Proposition \ref{prop:lmfdb}.
\end{proof}

\section{Non-existence of \texorpdfstring{$\GSp_4(\bF_2)$}{GSp4F2bar}-valued representations (Thm D, Cor E)}\label{GSp4nontri}

We now tackle the more general case of representations $\br:G_\Q\lra \GSp_4(\bF_2)$.

Our first result shows that the image of such a representation must be large:
\begin{theoremD}
  Let $\br:G_\Q\lra \GSp_4(\bF_2)$ be an irreducible $2$-ramified representation.
  Then $\im(\br)$ is either isomorphic to the Suzuki group $\Sz(\F_{2^r})$ for some odd $r>1$, or is equal to $\Sp_4(\F_{2^s})$ for some $s>1$.

  Hence there exist no irreducible $2$-ramified representations $\br:G_\Q\lra \GSp_4(\bF_2)$ such that $|\im(\br)|<29120$.
\end{theoremD}
We include the following well-known result for completeness:
\begin{lemma}\label{tri-char}Let $L/\Q$ be a quadratic extension whose class number $|\Cl(L)|$ is a power of $2$. If $2$ is not inert in $O_L$, then any finite abelian extension of $L$
  unramified outside $2$ has degree $2^n$ for some integer $n\ge 0$.
  In particular, there exist no non-trivial $2$-ramified characters  $\br:G_L\lra \bF^\times_2$.
\end{lemma}
\begin{proof}
  Let $K/L$ be a finite abelian $2$-ramified extension.
  By class field theory,
  there exists an ideal $I\subset O_L$ dividing a power of the ideal $2O_L$ such that $K$ is the ring class field with modulus $I$. Letting $H$ denote the
  Hilbert class field of $L$, we have the natural surjection $\Gal(K/L)\lra \Gal(H/L)$.
  Its kernel is described by class field theory (\cite[p.106, Proposition 6.114]{KKS}) as
  \begin{equation*}
    \coker\Big(\O^\times_L\stackrel{\alpha}{\lra} \Big(\bigoplus_{v:\text{real place}}\R^\times/\R^\times_{>0}\Big)\oplus (\O_L/I)^\times \Big);
  \end{equation*}
  here $\alpha$ is the diagonal map whose components are induced from
  the natural maps $\O^\times_L\lra (\O_L/I)^\times$  and
  $\O^\times_L\lra L^\times_v=\R^\times$ if $v$ is a real place while
  we ignore the complex places.
  By assumption, the orders of both this group and the class group $\Gal(H/L)\simeq \Cl(L)$ are powers of $2$.
  Hence $[K:L]$ is a power of $2$.

  For the second claim, notice that $\im(\br)$ has odd order. Thus, the claim is obvious from
  the  previous result.
\end{proof}
\begin{remark}
  The claim of Lemma \ref{tri-char} is true for $L=\Q(\sqrt{-3})$.
  Although $2$ is inert in $\O_L$, the unit group $\O^\times_L\simeq \mu_6$ kills the $2$-primary part of the cokernel. Thus, the same argument works since the class number of $L$ is one.
\end{remark}

Before we tackle Theorem D, we recall the classification of the maximal subgroups $H$ inside $G_r:=\Sp_4(\F_{2^r}),\ r>1$ (see \cite[p.648]{DZ} or \cite[Section 7.2]{BHRD}):
\begin{enumerate}
  \item $H$ is reducible (over $\bF_2$), when viewed as a representation with values in $\GL_4(\bF_2)$ via the natural inclusion $H\subset G_r\subset \GL_4(\bF_2)$;
  \item $H=\SO^+(4,\F_{2^r}):=\{g\in \SL_4(\F_{2^r})\ |\ Q(gx)=Q(x)\}$ where
    $Q:\F^{\oplus 4}_{2^r}\lra \F_{2^r}$ is the quadratic map defined by $Q(x_1,x_2,x_3,x_4)=x_1x_4+x_2x_3$;
  \item $H=\SO^-(4,\F_{2^r}):=\{g\in \SL_4(\F_{2^r})\ |\ Q(gx)=Q(x)\}$ where
    $Q:\F^{\oplus 4}_{2^r}\lra \F_{2^r}$ is the quadratic map defined by $Q(x_1,x_2,x_3,x_4)=x_1x_4+x_2^2+\mu x_3^2$ for any $\mu\in\F_{2^r}$ such that $x^2+x+\mu$ is irreducible over $\F_{2^r}$;
  \item $H=\Sz(\F_{2^r})$, a Suzuki group (occurs only when $r$ is odd);
  \item $H=G_s$ with $s\mid r$.
\end{enumerate}
According to \cite[p.32, Theorem 1.6.22 and p.366, Theorem 7.3.3]{BHRD}:
\begin{align*}
  |G_r|             & =2^{4r} (2^r-1)^2 (2^r+1)^2 (2^{2r}+1), \\
  |\SO^+(\F_{2^r})| & =2^{2r+1} (2^r-1)^2,                    \\
  |\SO^-(\F_{2^r})| & =2^{2r+1} (2^r-1) (2^r+1),              \\
  |\Sz(\F_{2^r})|   & =2^{2r} (2^r-1) (2^{2r}+1).
\end{align*}
In particular, the orthogonal groups $\SO^\pm(\F_{2^r})$ are proper subgroups of $G_r$.

The maximal subgroups of the Suzuki group $\Sz(\F_{2^r})\subset G_r\subset  \GL_4(\F_{2^r})$ with
$r=2s+1>1$ odd are classified in
\cite[p.385, Table 8.16]{BHRD} as follows:
\begin{enumerate}
  \item[(Sz-1)] the cyclic group of order $2^r-1$;
  \item[(Sz-2)] $\Sz(\F_{2^r})\cap B$ where $B$ is the upper triangular Borel subgroup of $\Sp_4$;
  \item[(Sz-3)] the dihedral group of order $2(2^r-1)$;
  \item[(Sz-4)] $C_{2^r\pm 2^{s+1}+1}:_{2^r} C_4:=\langle g,h\ |\ g^{2^r\pm 2^{s+1}+1}=h^4=1,\
      hgh^{-1}=g^{2^r} \rangle$;
  \item[(Sz-5)] $\Sz(\F_{2^{r'}})$ with $r'|r$.
\end{enumerate}

\begin{proof}[Proof of Theorem D]
  Let $\br:G_\Q\lra \GSp_4(\bF_2)$ be an irreducible Galois representation unramified outside $2$.
  Put $G=\im(\br)$. As in the proof of Proposition \ref{irred_to_abs},
  we see that the similitude character of $\br$ is trivial. Hence $G\subset \Sp_4(\bF_2)$.
  Further, by Theorem C, we may assume $G\subset G_r$ for some minimal $r>1$.

  If $G=G_r$, we are done.

  Suppose then that $G\subsetneq G_r$ and let $H$ be a maximal subgroup of $G_r$ containing $G$.
  We consider the various cases in the classification of $H$ given above.

  \begin{enumerate}
    \item We can immediately exclude the reducible case since $\br$ is irreducible.

    \item $G\subseteq \SO^+(\F_{2^r})$.

      It is known (cf. \cite[Section 7.2.2]{BHRD}) that $\SO^+(\F_{2^r})$ has an index $2$ subgroup isomorphic to
      \begin{align*}
        \PGL_2(\F_{2^r})\times \PGL_2(\F_{2^r}) & =\PSL_2(\F_{2^r})\times \PSL_2(\F_{2^r}) \\
                                                & =\SL_2(\F_{2^r})\times \SL_2(\F_{2^r}).
      \end{align*}
      Thus, there exist an extension $L/\Q$ of degree at most $2$, and two Galois representations $\tau_i:G_L \lra \SL_2(\F_{2^r})\subset \GL_2(\F_{2^r})$ such that $\br|_{G_L}\simeq \tau_1\otimes \tau_2$.
      Since $\br$ is unramified outside $2$,
      \begin{equation*}
        L\in\left\{\Q,\Q(\sqrt{2}),\Q(\sqrt{-2})\right\}.
      \end{equation*}
      For each finite place $v\nmid 2$ of $L$, $(\tau_1\otimes \tau_2)(I_v)=\br(I_v)=\{1\}$, where $I_v$ is the inertia group of $v$.
      Thus, for each $g\in I_v$,
      \begin{equation*}
        \tau_1(g)=\diag(a,a),\quad \tau_2(g)=\diag(a^{-1},a^{-1})\quad\text{for some }a\in \F^\times_{2^r}.
      \end{equation*}
      But $\im(\tau_i)\subset \SL_2(\bF_{2^r})$, so we conclude that
      $\tau_i(I_v)=\{1\}$ for $i=1,2$. Hence $\tau_i$ is unramified outside $2$.
      Applying \cite{MT}, $\tau_i$ has to be reducible. Since $G_L$ is normal in $G_\Q$,
      by Clifford's theorem, $\br|_{G_L}$ is isomorphic to the direct sum of four one-dimensional
      representations of $G_L$ that are unramified outside $2$. However, all of them are the trivial representation of $G_L$ by Lemma \ref{tri-char}. Thus, $|\im(\br)|\le 2$ and $\br$ can not be irreducible over $\bF_2$. We conclude that this case cannot occur.

    \item $G\subseteq \SO_4^-(\F_{2^r})$.

      It is known (cf. \cite[Section 7.2.2]{BHRD}) that  $\SO^-(\F_{2^r})$ has an index $2$ subgroup isomorphic to
      \begin{equation*}
        \PGL_2(\F_{2^{2r}})=\PSL_2(\F_{2^{2r}})=\SL_2(\F_{2^{2r}}).
      \end{equation*}
      Thus,
      there exist an extension $L/\Q$ of degree at most $2$ unramified outside $2$ and a continuous group homomorphism $\br|_{G_L}:G_L\lra \GL_2(\F_{2^{2r}})$.
      As in the previous case, $\br$ cannot be irreducible over $\bF_2$.
      We conclude that this case cannot occur.

    \item $G\subseteq \Sz(\F_{2^r})$.

      In the above classification of the maximal subgroups of $\Sz(\F_{2^r})$,
      cases (Sz-1)--(Sz--3) cannot occur since $\br$ is irreducible over $\bF_2$.

      In case (Sz-4), there exists an extension $L/\Q$ of degree at most $2$ and unramified outside $2$ such that
      $\br(G_L)$ is contained in the dihedral subgroup of the  group in question.
      Thus, by a similar argument as before, $\br$ cannot be irreducible over $\bF_2$.

      The only possibility is therefore (Sz-5): $G=\Sz(\F_{2^r})$ for some odd integer $r$.
      But $G\subseteq \Sz(\F_{2})=C_5\rtimes C_4$ (the Frobenius group
      of order $20$) can be seen to be impossible by the same argument as above, since it contains the dihedral group $D_{10}$ of order $10$.
  \end{enumerate}

  The last claim of Theorem D follows from the fact that $|\Sz(\F_{2^3})|=29120$ and $|\Sp_4(\F_{2^2})|=979000$.
\end{proof}
Summing up, we have
\begin{corE}
  Assume GRH.
  There exist no irreducible $2$-ramified representations $\br:G_\Q\lra \GSp_4(\bF_2)$.
\end{corE}
\begin{proof}As in the proof of Theorem A, under GRH we have
  \begin{equation*}
    [\Q_{\br}:\Q]=|\im(\br)|<4800.
  \end{equation*}
  Thus, the claim follows from the last part of Theorem D.
\end{proof}

\appendix

\section{Computer calculations}\label{app-calc}

To facilitate the reproducibility of our results, we detail the calculations used in proofs in Sections \ref{gld} and \ref{GSp4tri}.

\subsection{Successive minima calculations (in SageMath)}\label{app-sage}

The SageMath function \texttt{allmin} given below takes a list of possible degrees of number fields, a prime $p$ and a value of the $p$-length, and returns the minimum over the list of the expression
\begin{equation*}
  \frac{1}{n}+\frac{1}{p-1}\sum_{i=1}^N \frac{1}{p^{m_i-1}}
\end{equation*}
that appears in the bound of Proposition \ref{prop:upper_bound_C}.
This minimum is then used according to the strategy described at the end of Section \ref{prelim}, in conjunction with Odlyzko's tables of discriminant bounds, as described in the proofs of Propositions \ref{prop:length2} and \ref{prop:length3}.

\begin{verbatim}
def allmin(nlst, p, length):
    themin = 100000
    theminat = 0
    for n in nlst:
        x = mymin(n, p, length)
        if x < themin:
            themin = x
            theminat = n
    return themin, theminat

def mymin(n, p, length):
    themin = 100000
    m = n.valuation(p)
    if length == 1:
        s = 1/((p-1)*p^(m-1)) + 1/n
        if s < themin:
            themin = s
    elif length == 2:
        for m1 in srange(1, m+1):
            s = 1/((p-1)*p^(m1-1)) + 1/((p-1)*p^(m-m1-1)) + 1/n
            if s < themin:
                themin = s
    elif length == 3:
        for m1 in srange(1, m+1):
            for m2 in srange(1, m-m1+1):
                s = 1/((p-1)*p^(m1-1)) + 1/((p-1)*p^(m2-1)) +\
                    1/((p-1)*p^(m-m1-m2-1)) + 1/n
                if s < themin:
                    themin = s
    return themin
\end{verbatim}

\subsection{Exhaustive group search (in Magma)}\label{app-magma}

The following calculations were used in the proof of Proposition \ref{abs-case}.

The Magma function \texttt{findgps} listed below returns the sequence of all conjugacy classes of subgroups of $S_6$ that have at least one absolutely irreducible $4$-dimensional representation defined over an extension of $\F_2$, together with the corresponding sequence of absolutely irreducible representations.

\begin{verbatim}
findgps := function()
  G := Sym(6);
  S := Subgroups(G);
  gps := [];
  mods := [];
  for c in S do
    h := c`subgroup;
    M := AbsolutelyIrreducibleModules(h, GF(2));
    if 4 in [Dimension(m) : m in M] then
      Append(~gps, h);
      Append(~mods, M);
    end if;
  end for;
  return gps, mods;
end function;
\end{verbatim}

The result of running this function is

\begin{verbatim}
Magma V2.28-15    Sun Apr 13 2025 10:00:51   [Seed = 766557115]
Type ? for help.  Type <Ctrl>-D to quit.
> load "findgps.magma";
Loading "findgps.magma"
> gps, irreps := findgps();
> gps;
[
    Permutation group acting on a set of cardinality 6
    Order = 20 = 2^2 * 5
        (1, 2, 6, 4)
        (1, 6)(2, 4)
        (1, 4, 2, 6, 3),
    Permutation group acting on a set of cardinality 6
    Order = 36 = 2^2 * 3^2
        (1, 6, 2, 3)(4, 5)
        (1, 2)(3, 6)
        (1, 2, 4)(3, 6, 5)
        (1, 4, 2)(3, 6, 5),
    Permutation group acting on a set of cardinality 6
    Order = 36 = 2^2 * 3^2
        (2, 6)
        (1, 4)(3, 6)
        (1, 5, 4)(2, 3, 6)
        (1, 4, 5)(2, 3, 6),
    Permutation group acting on a set of cardinality 6
    Order = 36 = 2^2 * 3^2
        (1, 3)(2, 5)(4, 6)
        (1, 4)(3, 6)
        (1, 5, 4)(2, 3, 6)
        (1, 4, 5)(2, 3, 6),
    Permutation group acting on a set of cardinality 6
    Order = 60 = 2^2 * 3 * 5
        (2, 5)(4, 6)
        (3, 5, 4),
    Permutation group acting on a set of cardinality 6
    Order = 60 = 2^2 * 3 * 5
        (2, 3)(5, 6)
        (1, 4, 2)(3, 5, 6),
    Permutation group acting on a set of cardinality 6
    Order = 72 = 2^3 * 3^2
        (1, 3)(2, 5)(4, 6)
        (2, 6)
        (1, 4)(3, 6)
        (1, 5, 4)(2, 3, 6)
        (1, 4, 5)(2, 3, 6),
    Permutation group acting on a set of cardinality 6
    Order = 120 = 2^3 * 3 * 5
        (1, 2)
        (1, 3)(2, 4, 6),
    Permutation group acting on a set of cardinality 6
    Order = 120 = 2^3 * 3 * 5
        (1, 2)(3, 5)(4, 6)
        (1, 2, 3, 6, 5, 4),
    Permutation group acting on a set of cardinality 6
    Order = 360 = 2^3 * 3^2 * 5
        (3, 6)(4, 5)
        (1, 5)(2, 3, 6, 4),
    Symmetric group acting on a set of cardinality 6
    Order = 720 = 2^4 * 3^2 * 5
        (1, 2, 3, 4, 5)
        (1, 2)(3, 5)(4, 6)
]
> mods;
[
    [
        GModule of dimension 1 over GF(2),
        GModule of dimension 4 over GF(2)
    ],
    [
        GModule of dimension 1 over GF(2),
        GModule of dimension 4 over GF(2),
        GModule of dimension 4 over GF(2)
    ],
    [
        GModule of dimension 1 over GF(2),
        GModule of dimension 2 over GF(2),
        GModule of dimension 2 over GF(2),
        GModule of dimension 4 over GF(2)
    ],
    [
        GModule of dimension 1 over GF(2),
        GModule of dimension 2 over GF(2),
        GModule of dimension 2 over GF(2),
        GModule of dimension 4 over GF(2)
    ],
    [
        GModule of dimension 1 over GF(2),
        GModule of dimension 2 over GF(2^2),
        GModule of dimension 2 over GF(2^2),
        GModule of dimension 4 over GF(2)
    ],
    [
        GModule of dimension 1 over GF(2),
        GModule of dimension 2 over GF(2^2),
        GModule of dimension 2 over GF(2^2),
        GModule of dimension 4 over GF(2)
    ],
    [
        GModule of dimension 1 over GF(2),
        GModule of dimension 4 over GF(2),
        GModule of dimension 4 over GF(2)
    ],
    [
        GModule of dimension 1 over GF(2),
        GModule of dimension 4 over GF(2),
        GModule of dimension 4 over GF(2)
    ],
    [
        GModule of dimension 1 over GF(2),
        GModule of dimension 4 over GF(2),
        GModule of dimension 4 over GF(2)
    ],
    [
        GModule of dimension 1 over GF(2),
        GModule of dimension 4 over GF(2),
        GModule of dimension 4 over GF(2),
        GModule of dimension 8 over GF(2^2),
        GModule of dimension 8 over GF(2^2)
    ],
    [
        GModule of dimension 1 over GF(2),
        GModule of dimension 4 over GF(2),
        GModule of dimension 4 over GF(2),
        GModule of dimension 16 over GF(2)
    ]
]
\end{verbatim}

We can then identify these subgroups up to isomorphism:

\begin{verbatim}
> [Order(g) : g in gps];
[ 20, 36, 36, 36, 60, 60, 72, 120, 120, 360, 720 ]
> [GroupName(g) : g in gps];
[ F5, S3^2, S3^2, C3:S3.C2, A5, A5, S3wrC2, S5, S5, A6, S6 ]
\end{verbatim}

Magma uses \texttt{F5} to denote the Frobenius group $F_{20}$ of order $5(5-1)$.

We also check which of these subgroups are transitive in $S_6$:

\begin{verbatim}
> [IsTransitive(g) : g in gps];
[ false, false, true, true, false, true, true, false, true, true,
  true ]
\end{verbatim}

\end{document}